\documentclass{birkjour}
 \newtheorem{thm}{Theorem}[section]

 \theoremstyle{definition}
 
 \theoremstyle{remark}
 \newtheorem{rem}[thm]{Remark}
 
 \numberwithin{equation}{section}

\usepackage{amsmath}
\usepackage{amssymb}

\begin{document}

%
%
%
%
%
%
%
%
%

\title[Sign-changing Liouville Equation]
      {A sign-changing Liouville Equation}

\author{Alejandro Sarria}

\address{
Department of Mathematics\\
University of Colorado at Boulder\\
Boulder, CO 80309-0395, USA
}

\email{alejandro.sarria@colorado.edu}

\author{Ralph Saxton}
\address{
Department of Mathematics\\
University of New Orleans\\
New Orleans, LA, 70148, USA
}

\email{rsaxton@uno.edu}
\subjclass{35B44, 35B10, 35B65, 35Q35, 35B40.}

\keywords{Sign changing weight Liouville equation, blow-up, global existence, Schwarzian derivative.}

\date{October 1, 2014}

\begin{abstract}
We examine periodic solutions to an initial boundary value problem for a Liouville equation with sign-changing weight. A representation formula is derived  both for singular and nonsingular boundary data, including data arising from fractional linear maps. In the case of singular boundary data we study the effects the induced singularity has on the interior regularity of solutions. Regularity criteria are also found for a generalized form of the equation. 
\end{abstract}

\maketitle
\section{Introduction}
\label{sec:intro}
In this article, we study regularity of periodic solutions to the initial boundary value problem
\begin{equation}
\label{eq:liouville}
\begin{cases}
\partial_{\alpha t}\ln u=f(\alpha)u,\,\,\,\,\,\,\,\,\,\,\,&\alpha\in(0,1),\,\,\,t>0,
\\
u(\alpha,0)=u_0(\alpha),\,\,\,\,\,\,\,\,\,\,\,\,&\alpha\in[0,1],
\\
u(0,t)=u(1,t)=g(t),\,\,\,\,\,\,\,\,\,\,\,\,\,&t\geq0
\end{cases}
\end{equation}
for given bounded continuous functions $f$ and $u_0>0$ with prescribed boundary data $g>0$. Moreover, replacing $u$ on the right-hand side of \eqref{eq:liouville}i) by an arbitrary nonnegative function $\mathcal{F}(u)\in C^1(0,+\infty)$ we also establish regularity criteria for the resulting generalization of \eqref{eq:liouville}. 
Note that for $u>0$ to be $\alpha$-periodic, integration of (\ref{eq:liouville})i) (or of its generalization) over $(0, 1)$ requires $f(\alpha)$ to have at least one zero in $(0, 1)$ and to change sign if it is not identically zero.
We will refer to (\ref{eq:liouville}) as an initial periodic-boundary value problem for the sign-changing Liouville equation. 

Two particular versions of   Liouville's equation
\begin{equation}
\label{eq:hyp}
\square \psi (x,\tau) +\epsilon e^{\psi}=0,\,\,\,\,\,\,\,\,\,\,\,\,\,\,\,\epsilon=\pm1
,\,\,\,\,\,\,\,\,\,\,\,\,\,\,\,\,\,\square\equiv\partial^2_{\tau}-\partial^2_x\end{equation}
occur in various applications ranging from plasma physics and field theoretical modeling to fluid dynamics. This has made both versions of (\ref{eq:hyp}) a frequent topic of  investigation. In particular cases, the equation can be interpreted as a model for a  self-interacting scalar field in two-dimensional space-time, whose  properties have been the subject of extensive study (see for instance \cite{Crowdy1}-\cite{P2}). Equation (\ref{eq:liouville})i) is  obtained from (\ref{eq:hyp})  on changing to characteristic coordinates $\tau=\alpha+t$ and $x=\alpha-t$ and  setting $\psi=\ln u$, with $f(\alpha)$ replacing $\epsilon$ in the resulting equation. Further, the elliptic  Liouville equation
\begin{equation}
\label{eq:ell}
\Delta\ln\phi=-K\phi^2,\,\,\,\,\,\,\,\,\,\,\,\,\,\,\,\Delta\equiv\partial^2_x+\partial^2_y,
\end{equation}
appears in the study of two-dimensional steady, incompressible Euler flow with $\phi=e^{\psi}$, for $\psi$ the stream function relating  vorticity to the  velocity field, (\cite{S1}, \cite{Crowdy2}). As in its hyperbolic counterpart, (\ref{eq:ell}) reduces to (\ref{eq:liouville})i) along curves $\alpha=x+iy$ and $t=x-iy$ for $f\equiv-\frac{K}{2}$ and $\phi^2=u$.

Amongst its many physical applications, (\ref{eq:liouville})i) is found in Riemannian geometry. On prescribing the Gaussian curvature $K(x,\tau)$ for a pseudo-metric $ds^2=g^{ij}dx_idx_j=e^{2v(x,\tau)}(dx^2-d\tau^2)$ in two-dimensional Minkowski space,  the function $v$  satisfies the relation 
$$K(x,\tau)=-e^{-2v(x,\tau)}\square v(x,\tau)$$
in isothermal coordinates $(x,\tau)$, (\cite{rogers}).
In the case of constant Gaussian curvature, the change to characteristic variables 
then leads to Liouville's equation in the form, (\cite{L1}),
$$\partial_{\alpha t}v=Ke^{2v}.$$
In a related setting, the case of sign-changing $K$ has been studied recently for an elliptic version of the equation, (\cite{Ruf1}). 

We note that the relevance of (\ref{eq:liouville})i) in the field of fluid dynamics is not limited to the hyperbolic and elliptic models (\ref{eq:hyp}) and (\ref{eq:ell}). Indeed, if we restrict ourselves to quantities $u$ having 
constant spatial mean, then for a prescribed periodic function $f$ and initial data $u_0\equiv1$, (\ref{eq:liouville})i) appears in the study of  classes of semi-bounded solutions to the three-dimensional incompressible Euler equations, (\cite{Okamoto,Wunsch,Sarria1}). In this context, $f$ controls the concavity of the components of the velocity field and $u$ represents the jacobian of the transformation associated to the particle trajectories in the fluid. Finally, we remark that the subsequent generalization \eqref{u}i) may also have applications in the study of bi-Hamiltonian equations such as the $\mu$Hunter-Saxton equation (\cite{Khesin}), which describes the orientation of highly inertial liquid crystal director fields in the presence of an external magnetic field.

The outline for the remainder of the paper is as follows. A general representation formula for periodic solutions to (\ref{eq:liouville}) is derived in \S\ref{sec:solution}, and certain aspects of the structure of discontinuous and singular solutions are examined in \S\ref{sec:singular}. In section \S\ref{subsec:smoothbdry},  $L^{\infty}$ 
boundedness is studied for smooth boundary data $g(t)$, with the effects of singular $g$ on  interior regularity of solutions being considered in \S\ref{subsec:singularbdry}. Finite-time blowup in  $L^p$ , $1\leq p<\infty$, is  discussed in \S\ref{subsec:lpsmooth}. Lastly, regularity criteria for a generalization of \eqref{eq:liouville} (see \eqref{u}) is established in \S\ref{sec:generalization}. The reader may then refer to \S\ref{sec:examples} for specific examples.

\section{The Representation Formula}
\label{sec:solution}
In this section we derive a representation formula for solutions to (\ref{eq:liouville}). We begin by noticing that if $y$ satisfies the associated problem
\begin{equation}
\label{eq:p1}
\begin{cases}
\partial_{\alpha t}\ln y=f(\alpha)u_0(\alpha)y,\,\,\,\,\,\,&\,\,t>0,
\\
y(\alpha,0)=1,\,\,\,\,\,\,\,\,\,\,\,\,\,\,\,\,\,\,\,\,\,\,\,\,\,\,\,\,\,\,\,\,\,&\alpha\in[0,1],
\end{cases}
\end{equation}
then 
\begin{equation}
\label{eq:truesol}
\begin{split}
u(\alpha,t)=u_0(\alpha)y(\alpha,t).
\end{split}
\end{equation}
Below, in formula (\ref{eq:ysoln}), we have used (\ref{eq:p1})   to establish a representation formula for $y(\alpha, t)$ in order to find $u(\alpha, t)$.

We first notice that on dividing (\ref{eq:p1})i) by $y$, differentiating in time, and using the calculus identity $\partial_t(z^{-1}\partial_{\alpha t}\ln z)=z^{-1}\partial_{\alpha}(\partial_{tt}\ln z - \frac{1}{2}(\partial_t\ln z)^2)$
this shows
\begin{equation}
\label{eq:eq0}
\partial_{\alpha}\mathcal{R}(\partial_t\ln y; t)=0
\end{equation}
where
\begin{equation}
\label{eq:eq2}
\mathcal{R}(v; t)\equiv\partial_t v-\frac{1}{2}v^2.
\end{equation}
Next, integrating (\ref{eq:eq0})  from $0$ to $\alpha$ (or from $1$ to $\alpha$)  and using (\ref{eq:truesol}) with  boundary data $u(0,t)=u(1, t)=g(t)$, it follows that
\begin{equation}
\label{eq:eq3}
\mathcal{R}(\partial_t\ln y; t)=\mathcal{R}\left(\frac{d\ln g}{dt}; t\right),
\end{equation}
where we adopt the (simplifying) assumption that $u_0(0)=1$ (and subsequently to be consistent, $g(0)=1$).
Now,  the Schwarzian derivative  of a function $w(t)$ may be defined in terms of $\mathcal{R}$ by
\begin{equation}
\mathcal{S}(w; t)=\mathcal{R}\left(\frac{d}{dt} \ln \frac{dw}{dt}; t\right)
\label{eq:eq5}
\end{equation}
and it has the property (\cite{CDO}) that 
$\mathcal{S}(p; t)=\mathcal{S}(q; t)$ if and only if
$p$ and $q$ are related by a fractional linear (M\"{o}ebius) transformation, 
$$q(t)=\frac{ap(t)+b}{cp(t)+d}.$$
It follows from \eqref{eq:eq3} and \eqref{eq:eq5} that if we set $y=\partial_tY$ and 
\begin{equation}
\label{eq:G}
g(t)=\frac{dG}{dt},
\end{equation}
with $Y(\alpha, 0)=G(0)=0$, then
\begin{equation}
Y(\alpha, t)=\frac{a(\alpha)G(t)}{c(\alpha)G(t)+d(\alpha)}
\end{equation}
for some functions $a(\alpha), c(\alpha)$ and $d(\alpha)$. If we next set $\Delta(\alpha)=a(\alpha)d(\alpha)$, we find by successive differentiation that
\begin{equation}
\label{eq:y1}
y(\alpha, t)=\frac{\Delta(\alpha)g(t)}{(c(\alpha)G(t)+d(\alpha))^2}
\end{equation}
and 
\begin{equation}
\label{eq:preceding}
\partial_ty(\alpha, t)=\Delta(\alpha)\frac{(c(\alpha)G(t)+d(\alpha))\dot{g}(t)-2c(\alpha)g^2(t)}{(c(\alpha)G(t)+d(\alpha))^3}.
\end{equation}
Consequently \eqref{eq:p1}ii) and \eqref{eq:y1} imply $\Delta(\alpha)=d^2(\alpha)$. Integrating \eqref{eq:p1}i) now gives
\begin{equation}
\partial_t\ln y(\alpha, t)=\frac{d\ln g}{dt}(t)+\psi(\alpha, t)
\end{equation}
where 
\begin{equation}
\label{eq:psi}
\psi(\alpha, t)=\int_0^\alpha f(z)u_0(z)y(z, t)\,dz=\int_0^\alpha f(z)u(z, t)\,dz,
\end{equation}
 and so at $t=0$,
\begin{equation}
\label{eq:psi0}
\partial_ty(\alpha, 0)=\dot{g}(0)+\psi_0(\alpha)
\end{equation}
where 
\begin{equation}
\label{eq:Psi}
\psi_0(\alpha)=\int_0^\alpha f(z)u_0(z)\,dz. 
\end{equation}

In contrast, equation \eqref{eq:preceding} implies
\begin{equation}
\label{eq:here}
\partial_ty(\alpha, 0)=\Delta (\alpha)\frac{d(\alpha)\dot{g}(0)-2c(\alpha)}{d^3(\alpha)}=\dot{g}(0)-2\frac{c(\alpha)}{d(\alpha)}
\end{equation}
which lets us  combine equations \eqref{eq:y1}, \eqref{eq:psi0} and \eqref{eq:here}   to write $y(\alpha, t)$  in terms of  initial and boundary data as
\begin{equation}
\label{eq:ysoln}
y(\alpha, t)=\frac{g(t)}{(1-\frac{1}{2}\psi_0(\alpha)G(t))^2},
\end{equation}
 giving
\begin{equation}
\label{eq:finalsolution}
u(\alpha, t)=\frac{u_0(\alpha)g(t)}{(1-\frac{1}{2}\psi_0(\alpha)G(t))^2}
\end{equation}
from \eqref{eq:truesol}.

\begin{rem}
We note that if we choose constant boundary data $g(t)\equiv1$  the final solution simplifies to
\begin{equation}
\label{eq:finalsolutionspecial}
u(\alpha,t)=\frac{u_0(\alpha)}{(1-\frac{t}{2}\psi_0(\alpha))^2},
\end{equation} 
which, clearly, will persist for all time if $\psi_0(\alpha)\leq0$, for all $\alpha\in[0,1]$, but will become singular in finite time provided $\psi_0(\alpha)>0$ for some $\alpha\in (0, 1).$
More generally, since $g(t)>0$, monotonicity of $G(t)$ implies  $G(t)\rightarrow G_\infty$ as $t\rightarrow\infty$, where $G_\infty>0$ may, or may not, be bounded. Equation \eqref{eq:finalsolution} then shows solutions persist for all time  if ${max}_{\alpha\in (0, 1)}\psi_0(\alpha)<2/G_\infty$ while finite time blowup takes place if $\psi_0(\alpha)>2/G_\infty$ for some $\alpha\in(0, 1).$ \end{rem}

\section{Basic Properties of Singular Solutions}
\label{sec:singular}

Here we briefly examine some possible types of nonsmooth structure of  solutions from the formula given by (\ref{eq:finalsolution}). If we allow  jumps in $u(\alpha ,t)$ to be defined by
$$
[u(\alpha,\cdot)](t) = \lim_{\tau\downarrow t}u(\alpha, \tau) - \lim_{\tau\uparrow t}u(\alpha, \tau)
$$
and
$$
[u(\cdot , t)](\alpha) = \lim_{\beta\downarrow\alpha}u(\beta, t)- \lim_{\beta\uparrow\alpha}u(\beta ,t),
$$
then  jump discontinuities resulting from jumps in the boundary or initial data functions $g(t)$ or $u_0(\alpha)$, propagate along characteristics (lines of constant $t$ or $\alpha$), since    continuity of the primitive functions $G(t)$ and $\psi_0(\alpha)$  implies
$$
[u(\alpha,\cdot)](t)=\frac{u_0(\alpha)\, [g(\cdot)](t)}{(1-\frac{1}{2}G(t) \psi_0(\alpha))^2}
$$
and
$$
[u(\cdot , t)](\alpha)=\frac{[u_0(\cdot)](\alpha) \, g(t)}{(1-\frac{1}{2}G(t)\psi_0(\alpha))^2}.
$$
On  requiring $u_0(\alpha)$ and $g(t)$ to be positive, jumps in $u(\alpha, t)$, which stem from initial or boundary data,  remain  nonzero  along  corresponding characteristics.

If we next denote the set,\, $\Sigma$\,, by
\begin{equation}
\label{eq:Sigma}
\Sigma = \{(\sigma, \tau)\in(0, 1)\times(0, \infty): \psi_0(\sigma)\,G(\tau)=2\},
\end{equation}

then, to be strictly valid, the solution formula \eqref{eq:finalsolution}  requires that both a vertical and at least one horizontal characteristic  avoid intersecting $\Sigma$ at any point $(\sigma, \tau)$ prior to reaching $(\alpha, t)$, (\cite{kichenassamy}). This is not always possible, but
the  method used to construct    formula \eqref{eq:ysoln} above remains {\em a posteriori} valid at $(\alpha, t)$ 
[using vertical characteristics together with  horizontal characteristics coming either from the left (lines with constant $t> 0$  which meet $\alpha=0$) or from the right (through a similar construction\footnote[1]{For the latter construction one  integrates instead from $1$ to $\alpha$ and reaches the same solution formula as before with $\psi_0(\alpha)$ replaced by $\int_1^\alpha f(\alpha)u_0(\alpha)d\alpha$. Compatibility of initial data with  \eqref{eq:liouville}i) implies $\int_0^1 f(\alpha)u_0(\alpha)d\alpha$ = 0.}   using   periodicity of boundary data and meeting  $\alpha=1$)] if one imposes appropriate conditions on the function $f(\alpha)$.

A  condition sufficient for formula \eqref{eq:finalsolution} to hold  everywhere beneath $\Sigma$  (both before and after solutions begin to develop singularities) is for $\psi_0(\alpha)$ to be positive on only a single, open, connected set in $(0, 1)$
on which it has a single maximum and no  local minima.
If one assumes that $f'(\alpha_0)<0$ wherever $f(\alpha_0)=0$ and $\psi_0(\alpha_0)>0$, then  this suffices   since $\psi'_0(\alpha)=f(\alpha)u_0(\alpha)$ and  $\psi_0(\alpha)$ is consequently convex down  at its extrema.

In order to consider  the properties of  non-characteristic curves in $\Sigma$ further, 
suppose in (\ref{eq:finalsolution}) that the functions $f(\alpha), \, g(t)>0$ and $u_0(\alpha)>0$ are continuous for $\alpha\in (0, 1)$ and  set $\mathcal{F}(\alpha, t)=G(t)\psi_0(\alpha)-2$. Then $\mathcal{F}_t(\alpha, t)\neq 0$  wherever  $\psi_0(\alpha)>0$ and,   by the implicit function theorem,   there exists a unique curve, 
$t=\tilde{t}(\alpha)$, in the local neighborhood of  any point $(\tilde{\alpha}, \tilde{t})$ where $\mathcal{F}(\tilde{\alpha}, \tilde{t})=0$, through which $\mathcal{F}(\alpha, \tilde{t}(\alpha))=0.$
For ${max}_{\alpha\in (0, 1)}\psi_0(\alpha)>2/G_\infty$,\,   $\tilde{t}(\alpha)\in\Sigma$ then lies    
in the region 
\begin{equation}
\label{eq:A}
\mathcal{A}=\{(\alpha, t): \psi_0(\alpha)\geq 0, t\geq 0\}
\end{equation}

and is given  by the  formula
\begin{equation}
\label{eq:t}
\tilde{t}(\alpha)=G^{-1}(2/\psi_0(\alpha)).
\end{equation}

We   define 
\begin{equation}
\label{eq:dA}\partial\mathcal{A}=\{(\alpha, t): \psi_0(\alpha)=0, t> 0\}
\cup\{(\alpha, t):\psi_0(\alpha)>0, \,t=0\}
\end{equation}
and
\begin{equation}
\label{eq:A+-}\mathcal{A}_\pm=\{(\alpha, t)\in\mathcal{A}: f(\alpha)\gtrless 0\}\,\,\mbox{and}\,\,
\mathcal{A}_0=\{(\alpha, t)\in\mathcal{A}: f(\alpha)=0\}.
\end{equation}
On differentiating, $\tilde{t}(\alpha)$ in (\ref{eq:t}) gives, for $\alpha\in (0, 1)$,
\begin{equation}
\label{eq:slope}
\tilde{t}_\alpha(\alpha)=-\frac{2f(\alpha)u_0(\alpha)}{g\circ G^{-1}(2/\psi_0(\alpha))\,\psi_0^2(\alpha)}
\end{equation}
and so the slope of $\tilde{t}(\alpha)$ is negative in $\mathcal{A}_+$, positive in $\mathcal{A}_-$, and  zero in $\mathcal{A}_0$.

We will be interested  subsequently in  singular curves, $t=\tilde{t}(\alpha)$, which may meet the boundary, $\partial\mathcal{A}$.  In general, if $\psi_0(\alpha)$ is continuous on $[0, 1]$ and $G(t)$ is continuous and bounded on $[0, \infty)$, then no curve in $\Sigma$ can meet $\partial\mathcal{A}$.
Points on $\partial\mathcal{A}$ either take the form   $(\alpha_\sharp, t_\sharp\, )$ where $\psi_0(\alpha_\sharp)=0$ and $G(t_\sharp)$ is unbounded, or $(\alpha_\flat, 0)$ where $\psi_0(\alpha_\flat)$ is unbounded and $G(0)=0$. 
 We will let $\mathcal{A}$ take the form of the set  $\mathcal{C}=\{(\alpha, t):\alpha_l\leq\alpha\leq\alpha_r, t\geq 0\}$ where $\psi_0(\alpha)> 0\,\, $ for every
 $\alpha\in(\alpha_l, \alpha_r)$ and $\psi_0(\alpha_l)=\psi_0(\alpha_r)=0$.
By assuming $f(\alpha)>0$ for $\alpha$ close to $0$, we can let $\alpha_l=0$,   with $\alpha_r\leq1$.  $\partial\mathcal{C}$ is defined in an analogous way to $\partial\mathcal{A}$.

Following from the definitions, we have that $\psi_0(0)=G(0)=0$, to which we  add  some further simplifying hypotheses, based on the choice of the coefficient $f(\alpha)$ and the data $u_0(\alpha)$ and $g(t)$, in that leading order behaviour is given by  
$$\mbox{(H1)}\,\,\psi_0\sim  \alpha^{a_0} \mbox{ as }\alpha\downarrow 0, \,\, \psi_0\sim |\alpha-\alpha_r|^{a_r} \mbox{ as }\alpha\uparrow\alpha_r ,
\mbox{ and }G\sim \, t^{b_0} \mbox{ as }t\downarrow 0$$
where $a_0, a_r, b_0 >0$ and the symbol $\sim$ will mean that the quotient of the two sides tends to a positive constant in the limit. Similarly, we will  assume  that in the  limits of $\alpha$ approaching $\alpha_\flat \in(0, \alpha_r)$ or  $t$~approaching $t_\sharp>0$, 
$$\mbox{(H2)}\qquad\qquad \psi_0(\alpha)\sim |\alpha - \alpha_\flat |^a  \mbox{ as }\alpha\rightarrow \tilde{\alpha}\mbox{ or }G(t)\sim |t- t_\sharp |^b \mbox{ as }t\rightarrow t_\sharp$$ 
where $a=a(\alpha_\flat),\, b=b(t_\sharp).$ 
If, under these assumptions, the curve $t=\tilde{t}(\alpha)$ connects to $\partial{\mathcal C}$ at $(\alpha_\flat, 0)$,
then  (\ref{eq:Sigma}) implies that
$$|\tilde{t}(\alpha)|^{b_0}|\alpha-\alpha_\flat |^a\sim c\mbox{ as }\alpha\rightarrow\alpha_\flat $$
for some generic constant, $c>0.$
If  $t=\tilde{t}(\alpha)$ connects to  $\partial{\mathcal C}$  at~$(0, t_\sharp)$, then
$$|t_\sharp-\tilde{t}(\alpha)|^b\alpha^{a_0}\sim c\mbox{ as }\alpha\downarrow 0$$ 
 
As a result, $a_0, b_0 > 0$ require either $b<0$, or $a<0$, for the curve to meet  $\partial\mathcal{C}$ at $t=0$, or $\alpha = 0$, respectively. In the latter case,  one must clearly also have
$$|t_\sharp-\tilde{t}(\alpha)|^b|\alpha-\alpha_r|^{a_r}\sim c \mbox{ as }\alpha\uparrow\alpha_r .$$ 
In the event that  $G(t)$ becomes unbounded only as $t\rightarrow\infty$,  the branches of $\Sigma$ are  asymptotic to $\alpha=0$ and $\alpha=\alpha_r$.

\section{Regularity Results}
\label{sec:regularity}
In this section we are concerned simply with finite-time blowup, or global existence in time, of (\ref{eq:finalsolution}). In \S\ref{subsec:smoothbdry}, we study the interior regularity of solutions that are smooth at the boundary for all time, while, in \S\ref{subsec:singularbdry}, the case of non-smooth $g(t)>0$ is considered. More particularly, for the former we establish criteria in terms of the sign-changing function $f$ and initial data $u_0$ leading to finite-time blowup or global-in-time solutions. Then, in \S\ref{subsec:singularbdry}, we examine the effects on the interior regularity of $u$ of boundary data $g(t)$ having a particular singular form, specifically, for some $0<t_b<+\infty$, $g(t)$ is taken to be smooth on $t\in[0,t_b)$ but $\lim_{t\uparrow t_b}g(t)=+\infty$. In this case, we find that under certain conditions the solution $u$ can in fact diverge somewhere in the interior at a time $0<t_*<t_b$.

First we define some terminology. Let $M_0$ denote the greatest value attained by $\psi_0(\alpha)$ at a finite number of locations $\overline\alpha_i\in[0,1],\, 1\leq i\leq n$, namely
\begin{equation}
\label{eq:M0}
M_0\equiv\max_{\alpha\in[0,1]}\psi_0(\alpha)=\psi_0(\overline\alpha_i),\qquad\quad 1\leq i\leq n.
\end{equation}
Notice that, since $\psi_0(0)=\psi_0(1)=0$, $M_0\in\mathbb{R}^+\cup\{0\}$.

\subsection{Smooth Boundary Data}
\label{subsec:smoothbdry}

Suppose $u_0(\alpha)>0$ and $f(\alpha)$ are bounded continuous functions for all $\alpha\in[0,1]$, and the boundary data $g(t)>0$ is smooth. In this section we examine $L^{\infty}(0,1)$ regularity of (\ref{eq:finalsolution}). We begin by establishing simple criteria leading to global-in-time solutions.

\subsubsection{Global-in-time Solutions}
\label{subsubsec:globalsmoothbdry}
\hfill

\noindent
Note that a solution to (\ref{eq:liouville}) will persist for all time as long as (\ref{eq:finalsolution}) remains both finite and positive for all $0<t<+\infty$. Suppose $u_0$ and $f$ are such that $M_0=0$, that is $\psi_0(\alpha)\leq0$ for all $\alpha\in[0,1]$. This implies that $1-\frac{1}{2}G(t)\psi_0(\alpha)\geq1$ or, from (\ref{eq:finalsolution}),
\begin{equation}
\label{eq:global1}
0<u(\alpha,t)\leq g(t)u_0(\alpha).
\end{equation}
Since $g(t)$ is smooth for all time and $u_0(\alpha)\in L^{\infty}[0,1]$, $u$ remains finite for all $\alpha\in[0,1]$ and $0<t<+\infty$. This leads to Theorem \ref{thm:globaltheorem} below.

\begin{thm}
\label{thm:globaltheorem}
Consider the initial boundary value problem (\ref{eq:liouville}) for smooth boundary data $g(t)>0$. Suppose both the initial data $u_0$ and sign-changing function $f$ are continuous and $u_0(\alpha)\in L^{\infty}(0,1)$. If $u_0$ and $f$ are such that $\psi_0(\alpha)\leq0$ for all $\alpha\in[0,1]$, then $0<\left\|u(\cdot,t)\right\|_{\infty}<+\infty$ for all time. Moreover, the result still holds in the case where $M_0>0$ as long as $g$ is such that 
\begin{equation}
\label{eq:noblowcond}
\lim_{t\to+\infty}G(t)<\frac{2}{M_0}.
\end{equation}
\end{thm}
\begin{rem}
For $f(\alpha),u_0(\alpha)\in C^1[0,1]$, (\ref{eq:Psi}) implies that a sufficient condition for global solutions is that, for all $\alpha\in[0,1]$,
\begin{equation}
\label{eq:globalcond1}
f(\alpha)u_0'(\alpha)+f'(\alpha)u_0(\alpha)>0.
\end{equation}
Indeed, the above is equivalent to $\psi_0''>0$, which by $\psi_0(0)=\psi_0(1)=0$ implies that $M_0=0$. The reader may refer to \S\ref{sec:examples} where an example of a global solution is obtained for the case $f(\alpha)=2\alpha-1$ and $u_0(\alpha)\equiv1$. Also in \S\ref{sec:examples} we discuss a simple class of boundary data (a family of fractional linear maps) for which global solutions may be obtained for any choice of $u_0$ and $f$.
\end{rem}

\subsubsection{Finite-time $L^{\infty}$ Blowup}
\label{subsec:blowupsmoothbdry}
\hfill

\noindent
We now study finite-time blowup of solutions to \eqref{eq:liouville}. For smooth boundary data $g(t)>0$, suppose $u_0$ and $f$ are such that $M_0\in\mathbb{R}^+$, and assume $g$ is such that \eqref{eq:G} satisfies
\begin{equation}
\label{eq:blowcond}
\frac{2}{M_0}<\lim_{t\to+\infty}G(t)\leq +\infty.
\end{equation}
Since $g>0$ and $\dot G(t)=g(t)$, then by continuity \eqref{eq:blowcond} implies the existence of a finite $t_*>0$ such that 
\begin{equation}
\label{eq:G*}
G(t_*)=\lim_{t\uparrow t_*}\int_0^{t}{g(s)\,ds}=\frac{2}{M_0}.
\end{equation}
More particularly, from \eqref{eq:finalsolution} we see that
\begin{equation}
\label{eq:blow1}
\lim_{t\uparrow t_*}u(\overline\alpha_i,t)=+\infty,\qquad\quad 1\leq i\leq n
\end{equation}
for
\begin{equation}
\label{eq:t*}
t_*\equiv G^{-1}\left(2/M_0\right),
\end{equation}
and where $\overline\alpha_i$ denote the finite number of locations where $M_0$ is attained. On the contrary, if $\alpha\neq\overline\alpha_i$ (so that $\psi_0(\alpha)<M_0$), $u((\alpha,t)$ will converge, as $t\uparrow t_*$, to a finite positive constant $C(\alpha)$ given by
\begin{equation}
\label{eq:blow2}
C(\alpha)=g(t_*)u_0(\alpha)\left(1-\frac{\psi_0(\alpha)}{M_0}\right)^{-2}.
\end{equation}
We summarize the above results in Theorem \ref{thm:blowuptheorem} below.
\begin{thm}
\label{thm:blowuptheorem}
Consider the initial boundary value problem (\ref{eq:liouville}) for smooth boundary data $g(t)>0$. Suppose both the initial data $u_0$ and sign-changing function $f$ are continuous and $u_0(\alpha)\in L^{\infty}(0,1)$. If $u_0$ and $f$ are such that $\psi_0$ attains its greatest value $M_0>0$ at a finite number of points $\overline\alpha_i\in(0,1),\, 1\leq i\leq n$, and if $g$ is such that \eqref{eq:blowcond} holds, then there exists a finite $t_*>0$ (given by \eqref{eq:t*}) at which $u(\overline\alpha_i,t)$ diverges as $t\uparrow t_*$. In contrast, for $\alpha\neq\overline\alpha_i$\,, $u(\alpha,t)$ converges to the finite, positive constant in (\ref{eq:blow2}).
\end{thm}

\begin{rem}
\label{rem:suff}
If there exist\, $0<\alpha_0<\alpha_1\leq\alpha_2\leq1$ such that 
\begin{equation}
\label{suffblow1}
f(\alpha)u_0'(\alpha)\leq0,\quad f(\alpha_0)=0,\qquad \alpha\in[0,\alpha_1]
\end{equation}
and
\begin{equation}
\label{suffblow2}
(f(\alpha)u_0(\alpha))'\geq0,\quad f(\alpha_2)=0,\qquad \alpha\in[\alpha_1,1],
\end{equation}
then $M_0>0$. Consequently, \eqref{suffblow1}-\eqref{suffblow2} give a sufficient condition for $u(\overline\alpha_i,t)$ to blowup in finite time as long as \eqref{eq:blowcond} holds.
See \S\ref{sec:examples} for particular examples. 
\end{rem}

\subsection{Singular Boundary Data}
\label{subsec:singularbdry}
In this section we study the effects of singular boundary data on the interior regularity of (\ref{eq:finalsolution}). More particularly, suppose (\ref{eq:G}) has the form
\begin{equation}
\label{eq:Gsingular}
G(t)=\frac{1}{\beta}\left(\frac{1}{(1-t)^{\beta}}-1\right),\,\,\,\,\,\,\,\,\,\,\,\,\,\,\,\beta>0,
\end{equation}
so that
\begin{equation}
\label{eq:blowcond2}
\lim_{t\uparrow 1}G(t)=+\infty.
\end{equation}
Since $\dot G(t)=g(t)$, (\ref{eq:Gsingular}) then yields
\begin{equation}
\label{eq:gsingular}
g(t)=\frac{1}{(1-t)^{1+\beta}},\,\,\,\,\,\,\,\,\,\,\,\,\,\,\,\beta>0
\end{equation}
as the induced boundary blowup rate with boundary blowup time\footnote[2]{A simple rescaling argument shows that the choice $t_b=1$ is without loss of generality.}
\begin{equation}
\label{eq:tbdry}
t_b=1.
\end{equation}
We find that if $u_0$ and $f$ are such that $M_0>0$ (e.g. they satisfy the conditions in Remark \ref{rem:suff} above), then for all $\beta>0$, $u(\overline\alpha_i,t)$ blows up at a finite time $t_*$ satisfying $0<t_*<t_b$; whereas, for $\alpha\neq\overline\alpha_i$, $u$ remains bounded and positive. 

For the case $M_0=0$, define
\begin{equation}
\label{eq:omega}
\Omega\equiv\{\alpha\in[0,1]\,\,|\,\,\psi_0(\alpha)=0\},
\end{equation}
which note satisfies $\Omega\neq\emptyset$ due to $\psi_0(0)=\psi_0(1)=0$. Suppose $M_0=0$, so that $\Omega=\{\overline\alpha_i\}_{i=1}^n$. If $\alpha\in\Omega$, the induced boundary blowup time $t_b$ will represent the earliest blowup time for $u$; while, for $\alpha\notin\Omega$ (i.e. $\alpha\neq\overline\alpha_i$), the induced boundary blowup rate determines the behaviour of the solution (as well as its last configuration profile before blowup) as follows: If $\beta=1$, $u$ will stay both \textsl{finite} and \textsl{positive} as $t\uparrow t_b$, whereas, for $\beta\in(1,+\infty)$ or $\beta\in(0,1)$, $u$ will vanish or respectively diverge to $+\infty$ as $t\uparrow t_b$. Consequently, in the case $M_0=0$, we may refer to $\beta=1$ as a ``threshold'' exponent for the boundary singularity due to the drastic change in the last configuration profile of $u(\alpha,t)$ before blowup when $\beta=1\pm\epsilon$ for $\epsilon>0$ arbitrarily small. It turns out that the case $\beta=1$ corresponds to  \eqref{eq:Gsingular} being fractional linear, so that its Schwarzian derivative vanishes identically, i.e. $\mathcal{S}(G;t)\equiv0$. 

Now, using (\ref{eq:Gsingular}) and (\ref{eq:gsingular}) on (\ref{eq:finalsolution}), we obtain
\begin{equation}
\label{eq:usingular1.1}
u(\alpha,t)=\frac{4u_0(\alpha)}{\left(2-t(2+\psi_0(\alpha))\right)^2},\,\,\,\,\,\,\,\,\,\,\,\,\,\,\,\,\,\,\,\,\,\,\,\,\,\,\beta=1
\end{equation}
or
\begin{equation}
\label{eq:usingular1}
u(\alpha,t)=\frac{4\beta^2u_0(\alpha)\,(1-t)^{\beta-1}}{\left(2\beta(1-t)^{\beta}-\psi_0(\alpha)(1-(1-t)^{\beta})\right)^2},\quad\quad\beta\in\mathbb{R}^+\backslash\{1\},
\end{equation}
both of which imply 
\begin{equation}
\label{eq:int2}
u(\alpha,t)=\frac{u_0(\alpha)}{(1-t)^{1+\beta}},\,\,\,\,\,\,\,\,\,\,\,\,\,\,\,\,\,\,\,\alpha\in\Omega,\,\,\,\,\,\,\,\beta\in\mathbb{R}^+,
\end{equation}
so that $u$ diverges on $\Omega$ as $t$ approaches the induced boundary blowup time $t_b=1$. However, below we see how under certain conditions, $u$ may still diverge at an earlier time somewhere on $[0,1]\backslash\Omega$. 

\subsubsection{Boundary Singularity with $\beta=1$}
\hfill

\noindent
First we consider the simple case (\ref{eq:usingular1.1}), which corresponds to an induced boundary singularity of the form (\ref{eq:gsingular}) with $\beta=1$ . Since the earliest time $t_*$, satisfying $0<t_*\leq t_b$, for which  (\ref{eq:usingular1.1}) diverges is obtained from  
\begin{equation}
\label{eq:tsingular1}
t(\alpha)=\frac{2}{2+\psi_0(\alpha)},
\end{equation}
we find that
\begin{equation}
\label{eq:tsingular11}
t_*=
\begin{cases}
t_b=1,\,\,\,\,\,\,\,\,\,\,\,\,\,\,\,\,\,\,\,\,\,\,\,\,\,\,\,&M_0=0,
\\
\frac{2}{2+M_0}<t_b,\,\,\,\,\,\,\,\,&M_0>0.
\end{cases}
\end{equation}
Using the above on (\ref{eq:usingular1.1}) leads to the following result.

\begin{thm}
\label{thm:singular1}
Consider the initial boundary value problem (\ref{eq:liouville}) for  initial data $u_0(\alpha)$ and sign-changing function $f(\alpha)$ both continuous and bounded. For $\beta=1$, suppose the boundary data $g(t)$ has the singular form (\ref{eq:gsingular}) with prescribed boundary blowup time $t_b=1$. Then for $u_0$ and $f$ such that $M_0>0$ (see e.g. \eqref{suffblow1}-\eqref{suffblow2}), there exists $0<t_*<t_b$, given by (\ref{eq:tsingular11})ii), such that
\begin{equation}
\label{eq:usingular11}
\lim_{t\uparrow t_*}u(\alpha,t)=
\begin{cases}
+\infty,\,\,\,\,\,\,\,\,\,\,\,\,\,&\alpha=\overline\alpha_i,
\\
u_0(\alpha)\left(\frac{2+M_0}{M_0-\psi_0(\alpha)}\right)^2,\,\,\,\,\,\,\,\,&\alpha\neq\overline\alpha_i.
\end{cases}
\end{equation}
On the contrary, if $u_0$ and $f$ are so that $M_0=0$, then the earliest blowup time for $u$ is the induced boundary blowup time $t_b$, and
\begin{equation}
\label{eq:usingular12}
\lim_{t\uparrow t_b}u(\alpha,t)=
\begin{cases}
+\infty,\,\,\,\,\,\,\,\,\,\,\,\,\,&\alpha=\overline\alpha_i,
\\
\frac{4u_0(\alpha)}{\psi_0(\alpha)^2},\,\,\,\,\,\,\,\,&\alpha\neq\overline\alpha_i.
\end{cases}
\end{equation}
\end{thm}


\subsubsection{Boundary Singularity with $\beta\in\mathbb{R}^+\backslash\{1\}$}
\hfill

\noindent
Next we examine the instance (\ref{eq:usingular1}), which corresponds to singular boundary data of the form (\ref{eq:gsingular}) for $\beta\in\mathbb{R}^+\backslash\{1\}$. Here the earliest time $0<t_*\leq t_b$ for which (\ref{eq:usingular1}) diverges is obtained from  
\begin{equation}
\label{eq:tsingular25}
(1-t)^{\beta}=\frac{\psi_0(\alpha)}{2\beta+\psi_0(\alpha)}.
\end{equation}
This yields
\begin{equation}
\label{eq:tsingular2}
t(\alpha)=1-\left(\frac{\psi_0(\alpha)}{2\beta+\psi_0(\alpha)}\right)^{\frac{1}{\beta}}
\end{equation}
where, since we are only interested in times $0<t\leq t_b=1$, then $\alpha\in[0,1]$ in (\ref{eq:tsingular2}) must be such that 
$$\frac{\psi_0(\alpha)}{2\beta+\psi_0(\alpha)}\geq0;$$
otherwise the denominator in (\ref{eq:usingular1}) either never vanishes or does so at a time greater than $t_b$. 

As in the case $\beta=1$, (\ref{eq:tsingular2}) implies that $t_*=t_b=1$ whenever $\alpha\in\Omega$. Moreover, for $u_0$ and $f$ such that $M_0>0$, finite-time blowup occurs only in the interior (since $\psi_0(0)=\psi_0(1)=0$), more particularly,
\begin{equation} 
\label{eq:usingular21}
\lim_{t\uparrow t_*}u(\alpha,t)=
\begin{cases}
+\infty,\,\,\,\,\,\,\,\,\,\,\,\,\,\,\,\,\,\,\,&\alpha=\overline\alpha_i,
\\
C(\alpha),\,\,\,\,\,\,\,\,&\alpha\neq\overline\alpha_i
\end{cases}
\end{equation}
for positive constants $C(\alpha)$ given by
$$C(\alpha)=\frac{u_0(\alpha)(2\beta+M_0)^{1+\frac{1}{\beta}}M_0^{1-\frac{1}{\beta}}}{(M_0-\psi_0(\alpha))^2}$$
and $t_*>0$ satisfying
\begin{equation}
\label{eq:tsingular221}
t_*=1-\left(\frac{M_0}{2\beta+M_0}\right)^{\frac{1}{\beta}}<t_b=1.
\end{equation}
Lastly, if $u_0$ and $f$ are such that $M_0=0$, then for $\alpha\in\Omega=\{\overline\alpha_i\}_{i=1}^n$, (\ref{eq:int2}) holds and 
\begin{equation}
\label{eq:usingular00}
\lim_{t\uparrow t_b}u(\overline\alpha_i,t)\to+\infty,\,\,\,\,\,\,\,\,\,\,\,\,\,\,\,\,\,\,\,\beta\in\mathbb{R}^+\backslash\{1\}.
\end{equation}
In this case the boundary blowup time $t_b=1$ is also the earliest blowup time; however, for $\alpha\notin\Omega$, the behaviour of the solution varies relative to the value of $\beta$. Indeed, for $\alpha\not\in\Omega$ we have that $\psi_0(\alpha)<0$, and either $2\beta+\psi_0(\alpha)<0$ or $2\beta+\psi_0(\alpha)>0$. If the former, (\ref{eq:tsingular2}) yields $t(\alpha)>t_b$, which implies that the denominator of (\ref{eq:usingular1}) remains positive and finite for all $0<t\leq t_b$. As a result, the time-dependent term in the numerator gives 
\begin{equation}
\label{eq:usingular22}
\lim_{t\uparrow t_b}u(\alpha,t)=
\begin{cases}
0,\,\,\,\,\,\,\,\,\,\,\,\,\,\,\,\,\,\,\,&\beta\in(1,+\infty),\,\,\,\, \alpha\notin\Omega,
\\
+\infty,\,\,\,\,\,\,\,\,&\beta\in(0,1),\,\,\quad\,\, \alpha\notin\Omega.
\end{cases}
\end{equation}
If instead $2\beta+\psi_0(\alpha)>0$, so that $\frac{\psi_0}{2\beta+\psi_0}<0$, we use (\ref{eq:tsingular25}) and (\ref{eq:tsingular2}), as well as an argument similar to the one above, to show that $u$ diverges according to (\ref{eq:usingular00}) and (\ref{eq:usingular22}) above. We summarize these results in Theorem \ref{thm:singular2} below.

\begin{thm}
\label{thm:singular2}
Consider the initial boundary value problem (\ref{eq:liouville}) for continuous initial data $u_0(\alpha)\in L^{\infty}(0,1)$ and continuous sign-changing function $f(\alpha)$. Assume the boundary data $g(t)$ has the singular form (\ref{eq:gsingular}) for $\beta\in\mathbb{R}^+\backslash\{1\}$ and let $0<t_b<+\infty$ denote its induced blowup time. If $u_0$ and $f$ are such that $M_0>0$ (see e.g. \eqref{suffblow1}-\eqref{suffblow2}), then there exists $0<t_*<t_b$, given by (\ref{eq:tsingular221}), such that (\ref{eq:usingular21}) holds. In contrast, if $u_0$ and $f$ are so that $M_0=0$, then $u$ diverges according to (\ref{eq:usingular00})-(\ref{eq:usingular22}) as $t\uparrow t_*=t_b$ .
\end{thm}

The reader may refer to \S\ref{sec:examples} for specific examples.

\begin{rem}
From (\ref{eq:usingular00}) and (\ref{eq:usingular22}), both of which correspond to the case $M_0=0$, note that for $\beta\in(0,1)$, $u(\alpha,t)\to+\infty$ as $t\uparrow t_b$ everywhere in $[0,1]$. However, the singularity for $\alpha\in\Omega$ is ``stronger'' in the sense that, for $t_b-t>0$ small and $C\in\mathbb{R}^+$,
$$\frac{u(\alpha,t)}{u(\overline\alpha_i,t)}\sim C(t_b-t)^{2\beta},\,\,\,\,\,\,\,\,\,\,\,\,\,\,\,\,\,\,\,\alpha\neq\overline\alpha_i,$$
which vanishes as $t\uparrow t_b$.
\end{rem}

\subsection{Further $L^p(0,1)$ Regularity Results}
\label{subsec:lpsmooth}
In Theorem \ref{thm:blowuptheorem} of \S \ref{subsec:blowupsmoothbdry}, we established simple criteria, in terms of the initial and boundary data, as well as the sign-changing function $f$, for the existence of solutions to (\ref{eq:liouville}) which blowup in finite time in the $L^{\infty}(0,1)$ norm. In this section, we show that for $f$ in a large class of both smooth and non-smooth functions, $\left\|u(\cdot,t)\right\|_p\to+\infty$ as $t\uparrow t_*$ for all $1\leq p<+\infty$. More particularly, suppose $u_0$ and $f$ are such that $\psi_0(\alpha)$ attains a greatest positive value $M_0$ somewhere in $(0,1)$ and let $g(t)\in C^0[0,t_*]$ be such that \eqref{eq:blowcond} holds. Recall that $t_*>0$ denotes the finite $L^{\infty}(0,1)$ blowup time for $u$ satisfying \eqref{eq:G*}.
Now suppose there is $q\in\mathbb{R}^+$, $C_1\in\mathbb{R}^-$ and $r>0$ small, such that
\begin{equation}
\label{eq:local}
\psi_0(\alpha)\sim M_0+C_1\left|\alpha-\overline\alpha\right|^q
\end{equation}
for $0\leq\left|\alpha-\overline\alpha\right|\leq r$. To simplify subsequent arguments, we will assume that $M_0>0$ occurs at a single location $\overline\alpha\in(0,1)$. Moreover, in (\ref{eq:local}) we use the notation  
\begin{equation}
\label{eq:simexplanation}
k(\alpha) \sim L + h(\alpha), 
\end{equation}
valid for $0\leq|\alpha-\beta|\leq r$, to signify the existence of a function $l(\alpha)$ defined on $(\beta-r,\beta+r)$ such that
\begin{equation}
\label{eq:simexplanation2}
k(\alpha)-L=h(\alpha)(1+l(\alpha))\,\,\,\,\,\,\,\,\,\,\,\,\text{where}\,\,\,\,\,\,\,\,\,\,\,\,\,\,\lim_{\alpha\rightarrow\beta}l(\alpha)=0.
\end{equation}
Note that for $0<q<1$, (\ref{eq:local}) induces a ``cusp'' on the graph of $\psi_0$ at $\overline\alpha$, a ``kink'' if $q=1$, and various degrees of continuity on its derivatives at $\overline\alpha$ when $q>1$. Moreover, the boundedness of $u_0>0$, along with (\ref{eq:local}) and $\psi_0'=fu_0$, implies that
\begin{equation}
\label{eq:local2}
f(\alpha)\sim qC_1(\alpha-\overline\alpha)\left|\alpha-\overline\alpha\right|^{q-2}
\end{equation}
for $0\leq\left|\alpha-\overline\alpha\right|\leq r$. From (\ref{eq:local2}), observe that $f$ is continuous at $\overline\alpha$ if $q>1$, while a jump discontinuity of finite or infinite magnitude at this location will exist when $q=1$ or $0<q<1$, respectively. In any event, the solution formula (\ref{eq:finalsolution}) remains valid due to the integral term $\psi_0$ being continuous for all $\alpha\in[0,1]$ and $q>0$. 
We now establish the following Theorem.

\begin{thm}
\label{thm:lp}
Consider the initial boundary value problem (\ref{eq:liouville}) for smooth initial data $u_0(\alpha)$ and let $t_*>0$ denote the finite $L^{\infty}(0,1)$ blowup time for $u$ established in Theorem \ref{thm:blowuptheorem}. Suppose the sign-changing function $f(\alpha)$ satisfies (\ref{eq:local2}) for $1/2<q<+\infty$, while \eqref{eq:blowcond} holds for the prescribed boundary data $g(t)\in C^0[0,t_*]$. Further, let $u_0$ and $f$ be such that $\psi_0$ in \eqref{eq:Psi} attains its greatest value $M_0>0$ at a finite number of points $\alpha_i\in(0,1)$, $1\leq i\leq n$. Then $\left\|u(\cdot,t)\right\|_p\to+\infty$ as $t\uparrow t_*$ for all $1\leq p\leq+\infty$.
\end{thm}
\begin{proof}
Applying Jensen's inequality to (\ref{eq:finalsolution}) implies that
\begin{equation}
\label{eq:jensen}
\left\|u(\cdot,t)\right\|_p\geq\int_0^1{\frac{u_0(\alpha)g(t)}{\left(1-\frac{1}{2}G(t)\psi_0(\alpha)\right)^2}d\alpha}
\end{equation}
for $1\leq p<+\infty$. But from (\ref{eq:local}),
$$\epsilon+M_0-\psi_0(\alpha)\sim \epsilon+\left|C_1\right|\left|\alpha-\overline\alpha\right|^q$$
for $0\leq\left|\alpha-\overline\alpha\right|\leq r$ and $\epsilon>0$ small. Consequently, 
\begin{equation*}
\begin{split}
&\int_{\overline\alpha-r}^{\overline\alpha+r}{\frac{d\alpha}{(\epsilon+M_0-\psi_0(\alpha))^2}}\sim\int_{\overline\alpha-r}^{\overline\alpha+r}{\frac{d\alpha}{(\epsilon+\left|C_1\right|\left|\alpha-\overline\alpha\right|^q)^2}}
\\
&=\epsilon^{-2}\left[\int_{\overline\alpha-r}^{\overline\alpha}{\left(1+\frac{\left|C_1\right|}{\epsilon}\left(\overline\alpha-\alpha\right)^q\right)^{-2}d\alpha}+\int_{\overline\alpha}^{\overline\alpha+r}{\left(1+\frac{\left|C_1\right|}{\epsilon}\left(\alpha-\overline\alpha\right)^q\right)^{-2}d\alpha}\right].
\end{split}
\end{equation*}
Making the change of variables 
$$\sqrt{\frac{\left|C_1\right|}{\epsilon}}(\overline\alpha-\alpha)^{\frac{q}{2}}=\tan\theta,\,\,\,\,\,\,\,\,\,\,\,\,\,\,\sqrt{\frac{\left|C_1\right|}{\epsilon}}(\alpha-\overline\alpha)^{\frac{q}{2}}=\tan\theta$$
in the first and respectively second integral inside the bracket, we find that
\begin{equation}
\label{eq:general2}
\begin{split}
&\int_{\overline\alpha-r}^{\overline\alpha+r}{\frac{d\alpha}{(\epsilon+M_0-\psi_0(\alpha))^2}}\sim\frac{4\,\epsilon^{\frac{1}{q}-2}}{q\left|C_1\right|^{\frac{1}{q}}}\int_0^{\frac{\pi}{2}}{\frac{(\cos\theta)^{^{3-\frac{2}{q}}}}{(\sin\theta)^{^{1-\frac{2}{q}}}}d\theta}
\end{split}
\end{equation}
for small $\epsilon>0$. Suppose $q>1/2$. Then setting $\epsilon=\frac{2}{G}-M_0$ in (\ref{eq:general2}) implies that
\begin{equation}
\label{eq:general3}
\begin{split}
\int_{0}^{1}{\frac{d\alpha}{\left(1-\frac{1}{2}G(t)\psi_0(\alpha)\right)^2}}\sim C\left(G(t_*)-G(t)\right)^{^{\frac{1}{q}-2}}
\end{split}
\end{equation}
for $G(t_*)-G(t)>0$ small, $G(t_*)=\frac{2}{M_0}$ and $C\in\mathbb{R}^+$ given by
\begin{equation}
\label{eq:generalcst}
\begin{split}
C=\frac{8}{M_0^2}\left(\frac{M_0^2}{2\left|C_1\right|}\right)^{\frac{1}{q}}\Gamma\left(1+\frac{1}{q}\right)\Gamma\left(2-\frac{1}{q}\right)
\end{split}
\end{equation}
with $\Gamma\left(\cdot\right)$ the standard gamma function. We remark that the constant (\ref{eq:generalcst}) has been obtained via the identity
\begin{equation}
\label{eq:gammarel2}
\begin{split}
2\int_0^{\frac{\pi}{2}}{\frac{(\cos\theta)^{^{3-\frac{2}{q}}}}{(\sin\theta)^{^{1-\frac{2}{q}}}}d\theta}=q\,\Gamma\left(1+\frac{1}{q}\right)\Gamma\left(2-\frac{1}{q}\right),\,\,\,\,\,\,\,\,\,\,\,\,2>\frac{1}{q},
\end{split}
\end{equation}
which follows from well-known properties of the beta function. Then using (\ref{eq:general3}) on (\ref{eq:jensen}) yields
\begin{equation}
\label{eq:lower1}
\begin{split}
\left\|u(\cdot,t)\right\|_p\geq\int_0^1{\frac{u_0(\alpha)g(t)}{\left(1-\frac{1}{2}G(t)\psi_0(\alpha)\right)^2}d\alpha}\sim \frac{Cg(t_*)m_0}{\left(G(t_*)-G\right)^{2-\frac{1}{q}}}
\end{split}
\end{equation}
for $G(t_*)-G(t)>0$ small, $q>1/2$ and where, as a result of the boundedness and continuity of $u_0$ and $g$ for $\alpha\in[0,1]$ and respectively $t\in[0,t_*]$,  both $m_0=\min_{\alpha\in[0,1]}u_0(\alpha)$ and $g(t_*)$ are finite, positive constants. Taking the limit as $t\uparrow t_*$ (so that by continuity $G(t)\uparrow G(t_*)$) in (\ref{eq:lower1}) yields
\begin{equation}
\label{eq:lower2}
\begin{split}
\lim_{t\uparrow t_*}\left\|u(\cdot,t)\right\|_p=+\infty
\end{split}
\end{equation}
for all $1\leq p\leq+\infty$. 
\end{proof}

\section{A Generalized Sign-changing Liouville Equation}
\label{sec:generalization}
In this section we study regularity of solutions to the following generalization of \eqref{eq:liouville}:
\begin{equation}
\label{u}
\begin{cases}
\partial_{\alpha t}\ln u=f(\alpha)\mathcal{F}(u),\,\,\,\,\,\,\,\,\,\,\,&\alpha\in(0,1),\,\,\,t>0,
\\
u(\alpha,0)=u_0(\alpha),\,\,\,\,\,\,\,\,\,\,\,\,&\alpha\in[0,1],
\\
u(0,t)=u(1,t)=g(t),\,\,\,\,\,\,\,\,\,\,\,\,\,&t\geq0,
\end{cases}
\end{equation}
for $\mathcal{F}(z)$ an arbitrary differentiable function of $z$.
We will be particularly interested in the cases where the prescribed smooth boundary data $g(t)>0$ is either a non-decreasing function of time, $\dot g\geq0$, or a non-increasing one, $\dot g\leq0$. We begin by establishing the following blowup result for the former:

\begin{thm}
\label{general}
Consider the initial boundary value problem \eqref{u} for smooth initial data $u_0>0$ and smooth boundary data $g(t)>0$ satisfying $\dot g(t)\geq0$. Let $\alpha_0$ denote the first location in $(0,1)$ where the sign-changing function $f$ vanishes and assume there are positive constants $c$ and $d$ such that $\mathcal{F}(u)\in C^1(0,+\infty)$ satisfies
\begin{equation}
\label{ass}
0\leq\mathcal{F}(u),\quad\qquad c\mathcal{F}(u)\leq u\mathcal{F}'(u)\leq d\mathcal{F}(u)
\end{equation}
for $ '=\frac{d}{du}$. Then $u\to+\infty$ earliest at $\alpha=\alpha_0$ as $t$ approaches the finite time $t^*(\alpha_0)$ in \eqref{blowtime}.
\end{thm}
\begin{proof}
Let $\alpha_0$ be the first zero of $f(\alpha)$ in $(0,1)$. This is guaranteed to exist due to periodicity of $u$ and \eqref{ass}i), which imply that $\int_0^1{f(\alpha)\mathcal{F}(u)\,d\alpha}\equiv0$. More particularly, and without loss of generality, suppose 
\begin{equation}
\label{fass}
f(\alpha)
\begin{cases}
>0,\qquad &\alpha\in[0,\alpha_0),
\\
=0,\qquad\qquad &\alpha=\alpha_0,
\\
<0,\qquad &\alpha\in(\alpha_0,1].
\end{cases}
\end{equation}
Define
\begin{equation}
\label{H}
H(\alpha,t)\equiv(\ln u)_t\big|_0^{\alpha}=\frac{\dot u(\alpha,t)}{u(\alpha,t)}-\frac{\dot g(t)}{g(t)}
\end{equation}
and note that, as a result of \eqref{u}i) and \eqref{fass}, 
\begin{equation}
\label{Hpos}
H(\alpha,t)>0\qquad\quad \alpha\in(0,\alpha_0].
\end{equation}
Now, a straight-forward computation shows that $H$ satisfies
\begin{equation}
\label{Heq}
H_t(\alpha,t) = \int_0^\alpha f(x)\mathcal{F}'(u)u(x,t)H(x,t)\,dx + \frac{\dot{g}(t)}{g(t)}\int_0^\alpha f(x)\mathcal{F}'(u)u(x,t)\,dx.
\end{equation}
Suppose the boundary data is such that $\dot g\geq0$. Then using \eqref{ass}, \eqref{H} and  \eqref{Hpos}, and subsequently \eqref{u}i) on the right-hand side of \eqref{Heq}, we obtain
\begin{equation}
\label{Heq2}
H_t\geq\frac{c}{2}\left(\frac{\dot u}{u}\right)^2-\frac{c}{2}\left(\frac{\dot g}{g}\right)^2
\end{equation}
for $\alpha\in(0,\alpha_0]$. Since $\dot g/g\geq0$, \eqref{Heq2} then yields
\begin{equation}
\label{Heq3}
H_t\geq\frac{c}{2}H^2,\qquad\quad \alpha\in(0,\alpha_0],
\end{equation}
which we integrate to obtain
\begin{equation}
\label{Heq4}
0< \frac{1}{H(\alpha,t)}\leq\frac{1}{H_0(\alpha)}-\frac{c}{2}t,\qquad\quad \alpha\in(0,\alpha_0]
\end{equation}
where $H_0(\alpha)=H(0,\alpha)$. From \eqref{Heq4} and smoothness of $g$, we infer that $\dot u/u\to+\infty$ as $t$ approaches $t^*$ defined by
\begin{equation}
\label{tstar}
t^*\equiv\frac{2}{cH^*},\qquad\qquad H^*\equiv\max_{\alpha\in(0,\alpha_0]}H_0(\alpha).
\end{equation}
However, \eqref{u}i) implies that
\begin{equation}
\label{H0}
H_0(\alpha)=\int_0^{\alpha}{f(x)\mathcal{F}(u_0)\,dx}>0,\quad\qquad \alpha\in(0,\alpha_0],
\end{equation}
from which we conclude, by \eqref{fass}, that 
\begin{equation}
\label{blowtime}
t^*=\frac{2}{cH_0(\alpha_0)}.
\end{equation}
Thus $\dot u/u$ will diverge earliest, as $t\uparrow t^*$, at $\alpha=\alpha_0$. Lastly, for $t\in[0,t^*)$ and $\alpha\in(0,\alpha_0]$, \eqref{Heq4} implies that 
$$\partial_t\ln u\geq \partial_t\left[\ln g-\frac{2}{c}\ln\left(\frac{1}{H_0(\alpha)}-\frac{c}{2}t\right)\right],$$
which yields, upon integration and some simplification,
\begin{equation}
\label{blowincreasing}
u(\alpha,t)\geq\frac{g(t)u_0(\alpha)}{\left(1-\frac{c}{2}H_0(\alpha)t\right)^{2/c}}.
\end{equation}
From the above we infer that
$$\lim_{t\uparrow t^*}u(\alpha_0,t)=+\infty.$$
\end{proof}
Last we establish sufficient conditions for finite-time blowup or global-in-time existence of $u$ on $[0,\alpha_0]$ in the case $\dot g\leq0$.

\begin{thm}
\label{general2}
Consider the initial boundary value problem \eqref{u} for smooth initial data $u_0>0$ and smooth boundary data $g(t)>0$ satisfying $\dot g(t)\leq0$. Let $\alpha_0$ denote the first location in $(0,1)$ where the sign-changing function $f$ vanishes and assume there are positive constants $c$ and $d$ such that $\mathcal{F}(u)\in C^1(0,+\infty)$ satisfies \eqref{ass}. Then the following hold:
\begin{enumerate}
\item If $g$ is such that 
\begin{equation}
\label{gcond2}
\begin{split}
\lim_{t\to+\infty}\int_0^t{g(s)^dds}>\frac{2}{cH_0(\alpha_0)},
\end{split}
\end{equation}
then there exists a finite $t_*>0$ such that $u\to+\infty$ earliest at $\alpha=\alpha_0$ as $t\uparrow t_*$.
\vspace{0.1in}

\item If $g$ satisfies
\begin{equation}
\label{gcond3}
\begin{split}
\lim_{t\to+\infty}\int_0^t{g(s)^{c}ds}\leq\frac{2}{dH_0(\alpha_0)},
\end{split}
\end{equation}
then $u$ persists globally in time.
\end{enumerate}
\end{thm}
\begin{proof}
Without loss of generality we once again assume $f$ satisfies \eqref{fass}. First we obtain an upper bound for $u$. Since $\dot g\leq0$, using \eqref{u}i), \eqref{ass}, \eqref{fass} and \eqref{H}, on \eqref{Heq}, we obtain
\begin{equation}
\label{global0}
H_t(\alpha,t)\leq
c\frac{\dot g}{g}H(\alpha,t)+\frac{d}{2}H(\alpha,t)^2
\end{equation}
for $\alpha\in[0,\alpha_0]$. Rewriting the above as
\begin{equation}
\label{global1}
\partial_t\left(\frac{g^c}{H}\right)\geq-\frac{d}{2}g^c
\end{equation}
and integrating the latter yields
\begin{equation}
\label{ineqlower}
\frac{g(t)^c}{H(\alpha,t)}\geq\frac{1}{H_0(\alpha)}-\frac{d}{2}\int_0^t{g(s)^cds}
\end{equation}
for $\alpha\in(0,\alpha_0]$ and $H_0$ as in \eqref{H0}. In the above we also used the simplifying assumption $g(0)=1$. Multiplying both sides of \eqref{ineqlower} by $H$ and using \eqref{H}, we may integrate the resulting inequality in time to obtain
 \begin{equation}
\label{uppergen}
u(\alpha,t)\leq\frac{g(t)u_0(\alpha)}{\left(1-\frac{d}{2}H_0(\alpha)\int_0^t{g(s)^{c}ds}\right)^{2/d}}\,,
\end{equation}
which is valid on $[0,\alpha_0]$ and for as long as
\begin{equation}
\label{tinterval1}
\int_0^t{g(s)^{c}ds}<\frac{2}{dH_0(\alpha_0)}.
\end{equation}
Moreover, using \eqref{u}i), \eqref{ass}, \eqref{H} and $\dot g\leq0$ on \eqref{Heq}, leads to
\begin{equation}
\label{eqgen2}
\begin{split}
H_t(\alpha,t)\geq
d\frac{\dot g}{g}H(\alpha,t)+\frac{c}{2}H(\alpha,t)^2
\end{split}
\end{equation}
for all $\alpha\in[0,\alpha_0]$. Then proceeding as above we obtain
\begin{equation}
\label{eqgen3}
\begin{split}
0<\frac{g^d}{H(\alpha,t)}\leq\frac{1}{H_0(\alpha)}-\frac{c}{2}\int_0^t{g(s)^dds},
\end{split}
\end{equation}
from which we derive the lower-bound
\begin{equation}
\label{lowergen}
u(\alpha,t)\geq\frac{g(t)u_0(\alpha)}{\left(1-\frac{c}{2}H_0(\alpha)\int_0^t{g(s)^dds}\right)^{2/c}}.
\end{equation}
Inequality \eqref{lowergen} holds for all $\alpha\in[0,\alpha_0]$ and as long as
\begin{equation}
\label{tinterval2}
\int_0^t{g(s)^{d}ds}<\frac{2}{cH_0(\alpha_0)}.
\end{equation}
First suppose the smooth, non-increasing boundary data $g(t)>0$ is such that \eqref{gcond2} holds for some $d\in\mathbb{R}^+$. Then by continuity of $g$ there exists a finite $t_*>0$ such that 
\begin{equation}
\label{reqblow}
\lim_{t\uparrow t_*}\int_0^t{g(s)^dds}=\frac{2}{cH_0(\alpha_0)}
\end{equation}
and, thus,
$$\lim_{t\uparrow t_*}u(\alpha_0,t)=+\infty$$
by \eqref{lowergen}. This establishes the first part of the Theorem. Note that no conflict arises between this blowup and the upper-bound in \eqref{uppergen}. Indeed, since $d>c$ and $\dot g\leq0$ we have that
\begin{equation}
\label{condition}
\int_0^t{g(s)^dds}\leq\int_0^t{g(s)^{c}ds}\quad\qquad\text{and}\qquad\quad \frac{2}{dH_0(\alpha_0)}<\frac{2}{cH_0(\alpha_0)}.
\end{equation}
Consequently, \eqref{reqblow}, \eqref{condition} and continuity of $g$ imply the existence of $0<t_1<t_*$ such that the right-hand side of \eqref{uppergen} diverges as $t\uparrow t_1$. 

For the last part of the Theorem, suppose $g$ satisfies \eqref{gcond3}. Then \eqref{uppergen}, \eqref{lowergen} and \eqref{condition} imply that $u$ remains finite and positive for all $t\in\mathbb{R}^+$. This concludes the proof of the Theorem. 
\end{proof}
\begin{rem}
A simple example representative of the blowup result in Theorem \ref{general2} is given by $g(t)\geq e^{-kt}$ for $k\in\mathbb{R}^+$ fixed. In this case the right-hand side of \eqref{eqgen3} is bounded above by $1/H_0 -c (1-e^{-kdt})/2kd$, which reaches zero in finite time provided $H_0>2kd/c$.
\end{rem}

\section{Examples}
\label{sec:examples}

\subsection{Example 1 - Global Solution and Smooth Boundary Data}

Let $u_0(\alpha)\equiv1$, $f(\alpha)=2\alpha-1$ and $g(t)=2t+1$. Then (\ref{eq:Psi}) and (\ref{eq:G}) give
$$\psi_0(\alpha)=\alpha^2-\alpha,\,\,\,\,\,\,\,\,\,\,\,\,G(t)=t^2+t.$$
Note that $\psi_0\leq0$, so that $M_0=0$. Using (\ref{eq:finalsolution}), we obtain
\begin{equation}
\label{eq:globalex1}
u(\alpha,t)=\frac{2t+1}{\left(1-\frac{\alpha t}{2}(t+1)(\alpha-1)\right)^2}.
\end{equation}
The solution remains finite, and positive, for all $\alpha\in[0,1]$ and $0\leq t<+\infty$, whereas
\begin{equation}
\label{eq:globalex2}
\lim_{t\to+\infty}u(\alpha,t)=
\begin{cases}
0,\,\,\,\,\,\,\,\,\,\,\,\,\,\,&\alpha\in(0,1),
\\
+\infty,\,\,\,\,\,\,\,\,\,\,&\alpha\in\{0,1\}.
\end{cases}
\end{equation}
See Figure (\ref{fig:ex})(A) below.

\subsection{Example 2 - Finite-time Blow-up and Smooth Boundary Data}

Consider the same initial and boundary data as in Example 1, but now take $f(\alpha)=1-2\alpha$. Then (\ref{eq:Psi}) becomes $\psi_0(\alpha)=\alpha-\alpha^2$ with $M_0=1/4$ attained at $\overline\alpha_1=1/2$. The solution (\ref{eq:finalsolution}) now reads
\begin{equation}
\label{eq:blowex2}
u(\alpha,t)=\frac{2t+1}{\left(1-\frac{\alpha t}{2}(t+1)(1-\alpha)\right)^2}.
\end{equation}
Since $G(t)=t^2+t$, we solve $G(t)=8$ and find that $t_*=\frac{1}{2}\left(\sqrt{33}-1\right)\sim2.37.$ Using (\ref{eq:blowex2}) we conclude that 
\begin{equation}
\label{eq:blowex22}
\lim_{t\to t_*}u(\alpha,t)=
\begin{cases}
+\infty,\,\,\,\,\,\,\,\,\,\,\,\,\,\,&\alpha=\overline\alpha_1,
\\
\frac{\sqrt{33}}{(1-2\alpha)^4},\,\,\,\,\,\,\,\,\,\,&\alpha\neq\overline\alpha_1.
\end{cases}
\end{equation}
See Figure (\ref{fig:ex})(B) below.
\begin{center}
\begin{figure}[!ht]
\includegraphics[scale=0.34]{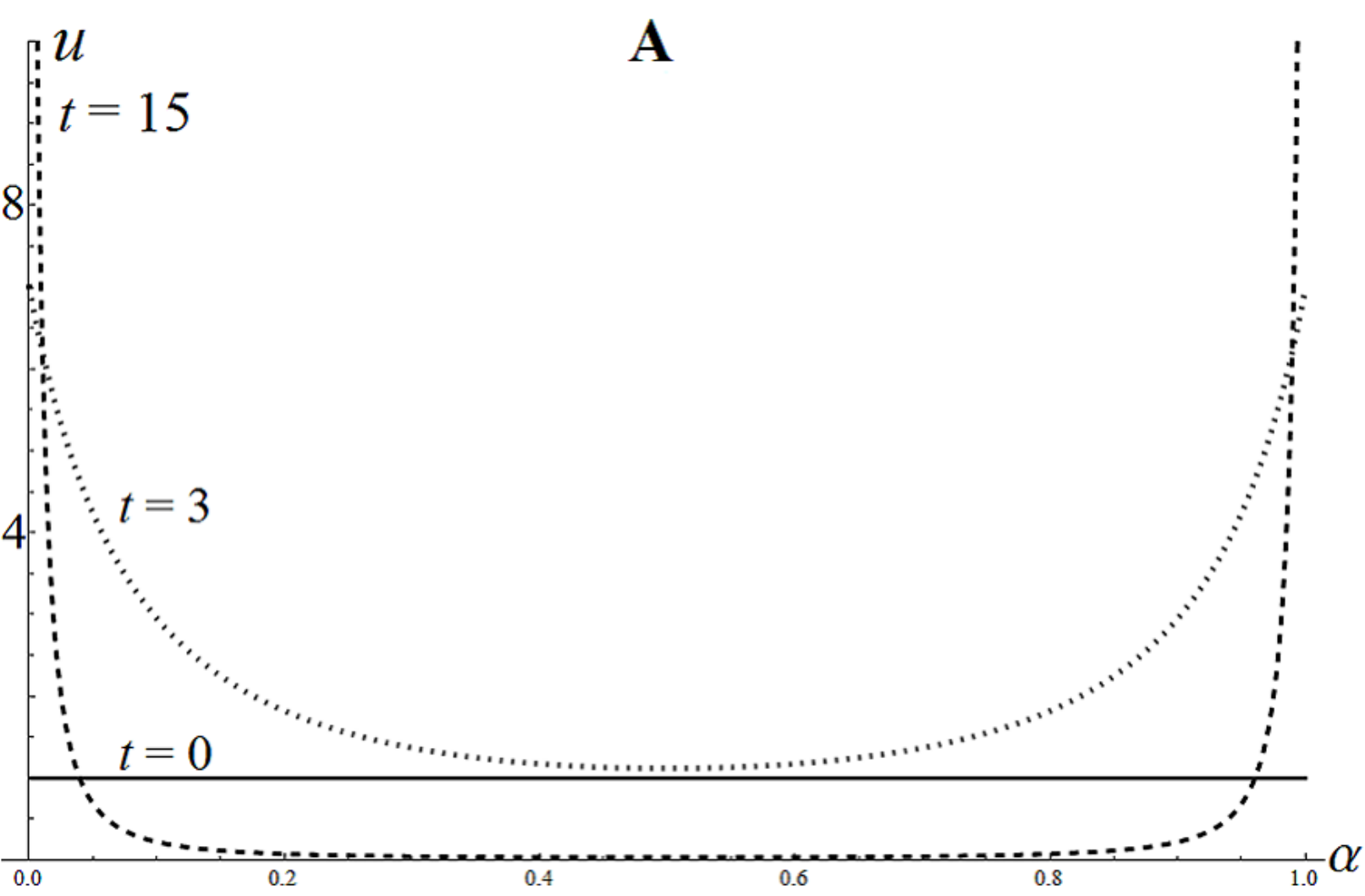} 
\includegraphics[scale=0.36]{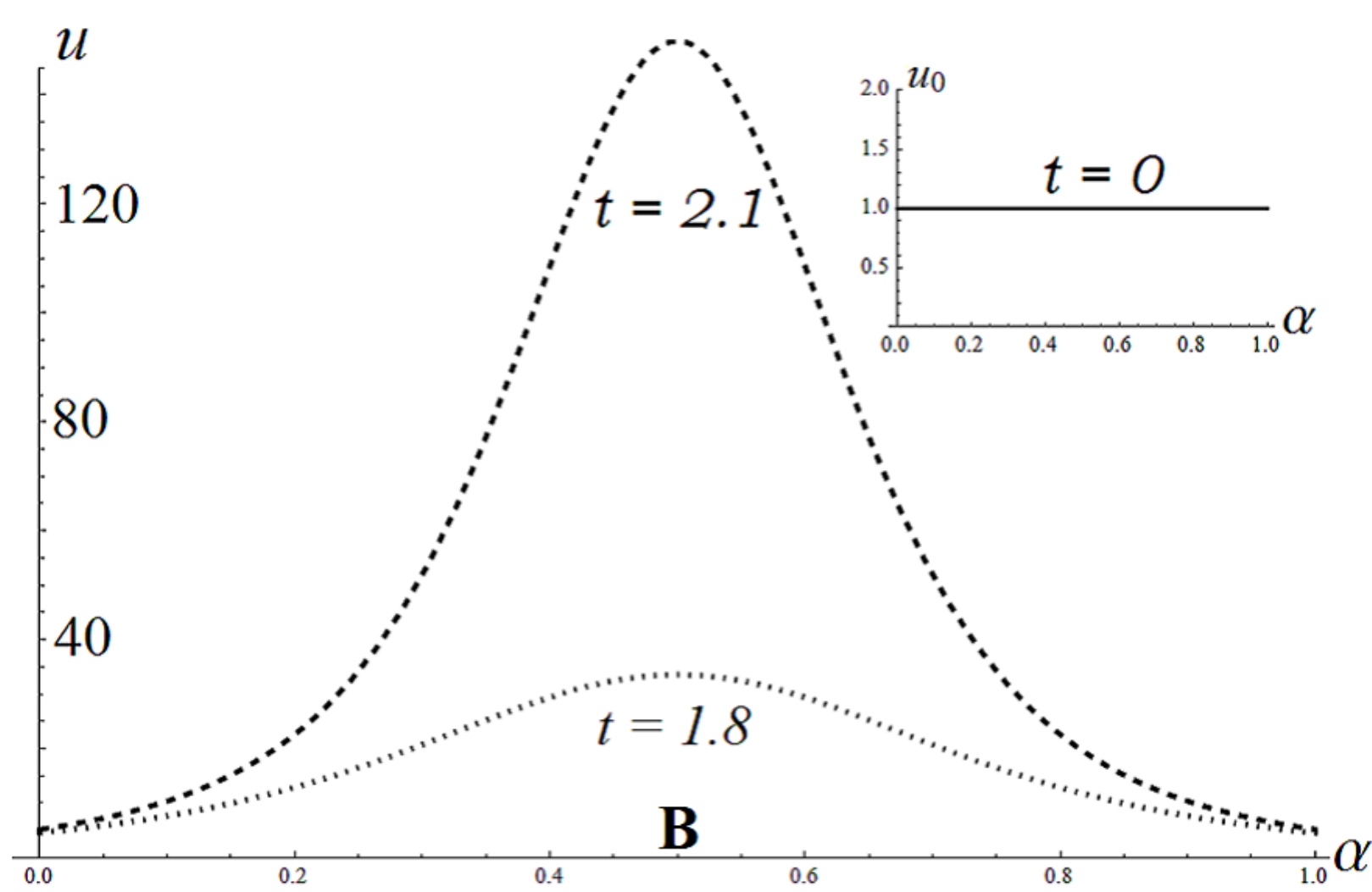} 
\caption{For Example 1 above, Figure A depicts the global-in-time behaviour of (\ref{eq:globalex1}), while B represents finite-time blowup of $u$ in (\ref{eq:blowex2}) as $t\uparrow t_*\sim2.37$.}
\label{fig:ex}
\end{figure}
\end{center}
The following two Examples are instances of Theorem \ref{thm:singular1}.

\subsection{Example 3 - Induced Boundary Blow-up for $M_0=0$}

For $u_0$ and $f$ as in Example 1, let $g(t)=(1-t)^{-2}$. This implies that $M_0=0$ occurs only at the boundary points. Then using (\ref{eq:usingular1.1}), we obtain
\begin{equation}
\label{eq:ex3}
u(\alpha,t)=\frac{4}{(2-t(\alpha^2-\alpha+2))^2},
\end{equation}
which diverges earliest, at the boundary points $\overline\alpha_i=\{0,1\}$, as $t\uparrow t_b=1$. In contrast, for $\alpha\in(0,1)$, $u$ remains finite (and positive) for all $t\in[0,1]$. In fact,
$$\lim_{t\uparrow 1}u(\alpha,t)=\frac{4}{(\alpha^2-\alpha)^2},\,\,\,\,\,\,\,\,\,\,\,\,\,\,\,\,\alpha\in(0,1).$$ See Figure \ref{fig:ex4}(A) below.

\subsection{Example 4 - Earlier Interior Blowup for $M_0>0$ with Singular $g(t)$}

Take $u_0$ and $f$ as in Example 2 and $g(t)=(1-t)^{-2}$. Then we now have $M_0=1/4$ occurring at $\overline\alpha_1=1/2$. From (\ref{eq:usingular1.1}), we obtain
\begin{equation}
\label{eq:ex4}
u(\alpha,t)=\frac{4}{(2+t(\alpha^2-\alpha-2))^2},
\end{equation}
which diverges earliest at $\alpha=\overline\alpha_1$ as  
$$t\uparrow t_*=\frac{2}{2+M_0}=\frac{8}{9}<t_b=1.$$ 
In contrast, for $\alpha\neq\overline\alpha_1$, $u$ remains finite and positive for all $t\in[0,t_*]$. In this case, the final solution profile is given by
$$\lim_{t\uparrow t_*}u(\alpha,t)=\left(\frac{9}{1-4\alpha+4\alpha^2}\right)^2,\,\,\,\,\,\,\,\,\,\,\,\,\,\,\,\,\alpha\in[0,1]\backslash\{\overline\alpha_1\}.$$ See Figure \ref{fig:ex4}(B) below.

\begin{center}
\begin{figure}[!ht]
\includegraphics[scale=0.32]{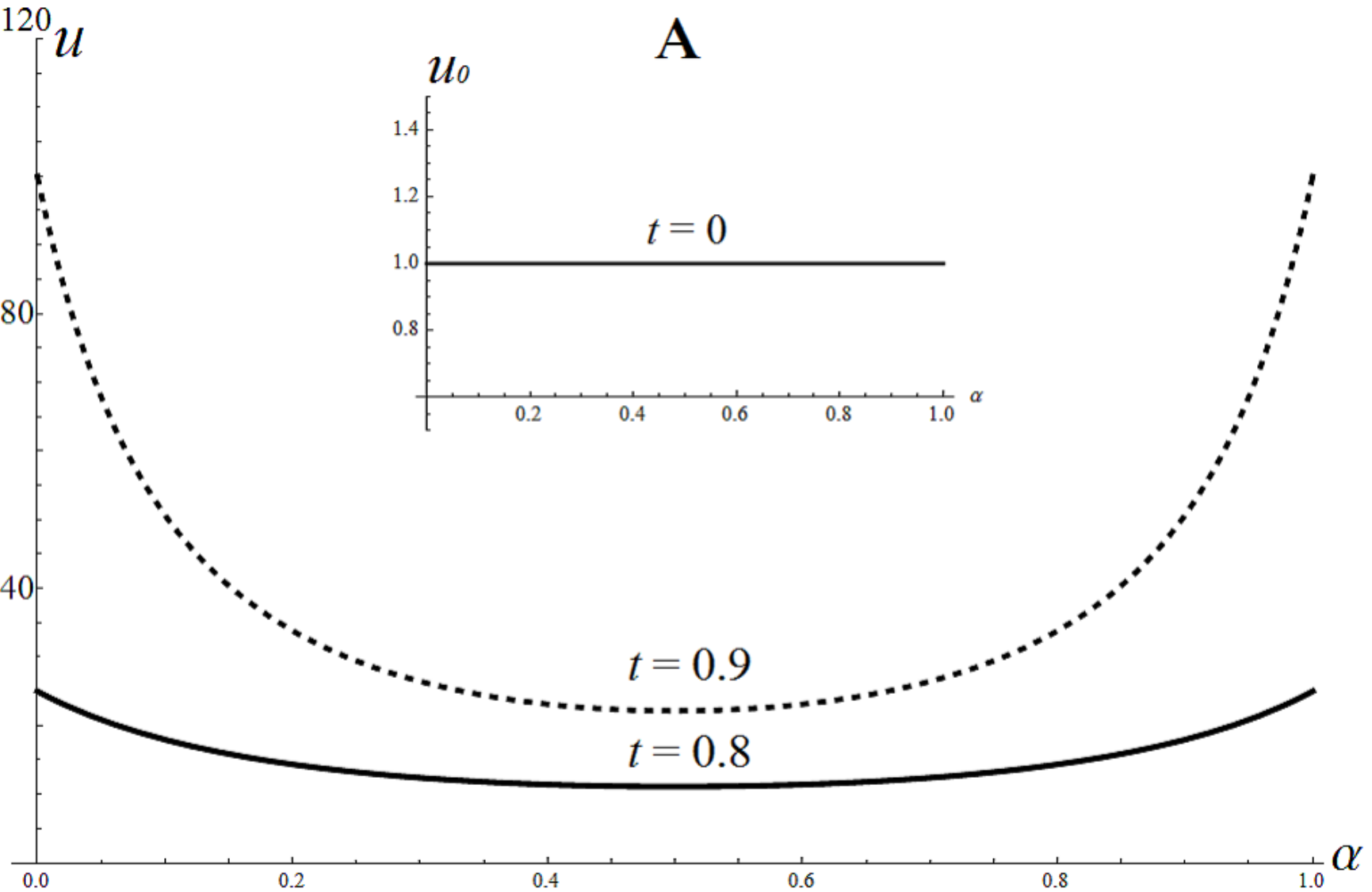} 
\includegraphics[scale=0.32]{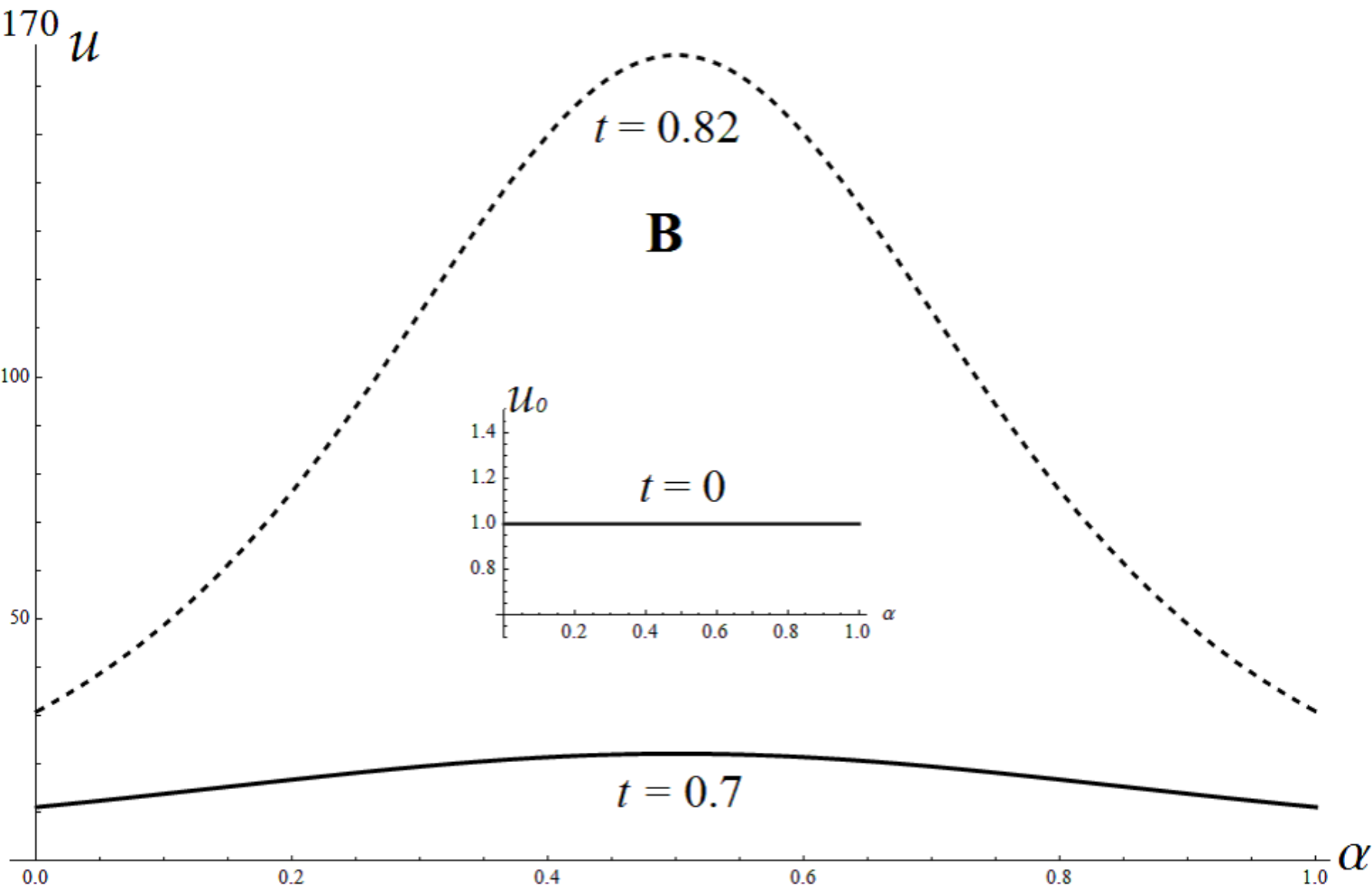} 
\caption{Blowup profiles for Examples 4 and 5 with singular boundary data. Figure A represents earliest blowup of (\ref{eq:ex3}) at the boundary as $t\uparrow t_b=1$ in the case where $M_0=0$, whereas, for $M_0>0$, Figure B depicts earliest blowup of (\ref{eq:ex4}) in the interior as $t\uparrow t_*<t_b$.}
\label{fig:ex4}
\end{figure}
\end{center}

\end{document}